\documentclass[11pt]{amsart}
\usepackage{amssymb,amsthm,amsmath}
\usepackage[bookmarks,colorlinks,breaklinks]{hyperref}
\usepackage{color}
\newtheorem{thm}{Theorem}

\newtheorem{prop}[thm]{Proposition}
\newtheorem{lem}[thm]{Lemma}

\newtheorem{defn}[thm]{Definition}

\theoremstyle{definition}

\newtheorem{rem}[thm]{Remark}

\numberwithin{thm}{section}
\numberwithin{equation}{thm}

\newcommand{\A}{\ensuremath{{\mathbb{A}}}}
\newcommand{\bP}{\ensuremath{{\mathbb{P}}}}

\renewcommand{\P}{\ensuremath{{\mathbb{P}}}}

\newcommand{\Qbar}{\ensuremath{\overline {\mathbb{Q}}}}
\newcommand{\R}{\ensuremath{{\mathbb{R}}}}

\newcommand{\N}{\ensuremath{{\mathbb{N}}}}
\newcommand{\p}{\ensuremath{{\mathfrak{p}}}}

\newcommand{\fp}{\ensuremath{{\mathfrak{p}}}}
\newcommand{\fq}{\ensuremath{{\mathfrak{q}}}}
\newcommand{\fm}{\ensuremath{{\mathfrak{m}}}}
\newcommand{\fo}{\ensuremath{{\mathfrak{o}}}}

\newcommand{\cL}{\mathcal L}
\newcommand{\cM}{\mathcal M}

\newcommand{\cO}{\mathcal O}

\newcommand{\cX}{\mathcal X}

\newcommand{\lra}{\longrightarrow}
\newcommand{\Kbar}{\ensuremath {\overline K}}

\newcommand{\ba}{{\boldsymbol \alpha}}
\newcommand{\bF}{{\boldsymbol f}}

\renewcommand{\O}{\ensuremath{{\mathcal{O}}}}

\DeclareMathOperator{\Gal}{Gal}

\DeclareMathOperator{\Aut}{Aut}

\begin{document}

\title[Iterated Galois groups of unicritical polynomials]{Finite index
  theorems for iterated Galois groups of unicritical polynomials}

\author[A. Bridy]{Andrew Bridy}
\address{Andrew Bridy\\Department of Mathematics\\ Yale University\\
New Haven, CT 06511 \\ USA}
\email{andrew.bridy@yale.edu}

\author[J. R. Doyle]{John R. Doyle}
\address{John R. Doyle \\ Department of Mathematics and Statistics \\ Louisiana Tech University \\ Ruston, LA 71272 \\ USA}
\email{jdoyle@latech.edu}

\author[D. Ghioca]{Dragos Ghioca}
\address{Dragos Ghioca \\ Department of Mathematics \\ University of British Columbia \\ Vancouver, BC V6T 1Z2 \\ Canada}
\email{dghioca@math.ubc.ca}

\author[L.-C. Hsia]{Liang-Chung Hsia}
\address{
Liang-Chung Hsia\\
Department of Mathematics\\
National Taiwan Normal University\\
Taipei, Taiwan, ROC
}
\email{hsia@math.ntnu.edu.tw}

\author[T. J. Tucker]{Thomas J. Tucker}
\address{Thomas J. Tucker\\Department of Mathematics\\ University of Rochester\\
Rochester, NY, 14620, USA}
\email{thomas.tucker@rochester.edu}



\date{}

\dedicatory{}

\begin{abstract}
  Let $K$ be the function field of a smooth, irreducible curve defined over $\Qbar$. Let $f\in K[x]$ be of the form $f(x)=x^q+c$ where $q = p^{r}, r \ge 1,$ is a power of the prime number $p$, and
  let $\beta\in \Kbar$. For all $n\in\mathbb{N}\cup\{\infty\}$, the
  Galois groups $G_n(\beta)=\Gal(K(f^{-n}(\beta))/K(\beta))$ embed
  into $[C_q]^n$, the $n$-fold wreath product of the cyclic group $C_q$.  
  We show that if  $f$ is not isotrivial, then
  $[[C_q]^\infty:G_\infty(\beta)]<\infty$ unless $\beta$ is postcritical or periodic.  We are also able to
  prove that if $f_1(x)=x^q+c_1$ and $f_2(x)=x^q+c_2$ are two such distinct polynomials, then 
  the fields $\bigcup_{n=1}^\infty K(f_1^{-n}(\beta))$ and $\bigcup_{n=1}^\infty K(f_2^{-n}(\beta))$ 
  are disjoint over a finite extension of $K$.  
\end{abstract}

\maketitle

\section{Introduction and Statement of Results}\label{intro}
Let $K$ be a field.  Let $f\in K(x)$ with $d=\deg f\geq 2$ and let
$\beta\in \P^1(\Kbar)$.  For $n\in\N$, let
$K_n(f,\beta)=K(f^{-n}(\beta))$ be the field obtained by adjoining the
$n$th preimages of $\beta$ under $f$ to $K(\beta)$. (We declare that
$K(\infty)=K$.)  Set
$K_\infty(f,\beta)=\bigcup_{n=1}^\infty K_n(f,\beta)$.  For
$n\in\N\cup\{\infty\}$, define
$G_n(f,\beta)=\Gal(K_n(f,\beta)/K(\beta))$.  In most of the paper, we
will write $G_n(\beta)$ and $K_n(\beta)$, suppressing the dependence
on $f$ if there is no ambiguity.

The group $G_\infty(\beta)$ embeds into $\Aut(T^d_\infty)$, the
automorphism group of an infinite $d$-ary rooted tree
$T^d_\infty$. Recently there has been much work on the problem of
determining when the index $[\Aut(T^d_\infty):G_\infty(\beta)]$ is
finite. The group $G_\infty(\beta)$ is the image of an arboreal Galois
representation, so this finite index problem is a natural analog in
arithmetic dynamics of the finite index problem for the $\ell$-adic
Galois representations associated to elliptic curves, resolved by
Serre's celebrated Open Image Theorem~\cite{Serre}. By work of Odoni~\cite{OdoniIterates}, 
one expects that a generically chosen rational function has a surjective arboreal representation, 
i.e., that $[\Aut(T^d_\infty):G_\infty(\beta)]=1$.

In this paper we study the family of polynomials $f(x)=x^d+c$ for $c\in K$, which up to change of variables represents 
all polynomials with precisely one (finite) critical point. If the field $K$ contains a primitive $d$th root of unity, 
then it is easy to show that for $f$ in this family, $G_\infty(\beta)$ sits in $[C_d]^\infty$, 
the infinite iterated wreath product of the cyclic group $C_d$ (with $d$ elements).
As $\Aut(T^d_n)\cong [S_d]^n$, this means that if $d\geq 3$, then $[\Aut(T^d_\infty):[C_d]^\infty]=\infty$. Thus
it is impossible for $G_\infty(\beta)$ to have finite index within this family (except when $d=2$). 
However, this simply means that, given the constraint on the size of $G_\infty(\beta)$, 
we should ask a different finite index question. We turn to the problem of when $G_\infty(\beta)$ 
has finite index in $[C_d]^\infty$. 

Before stating our main results, we set some notation. Throughout this paper, unless otherwise indicated, $K$ will refer to 
a function field of transcendence degree $1$ over its field of constants $\Qbar$.
In other words, $K$ is the function field of a smooth, projective, irreducible curve $C$ over $\Qbar$. We say that $f\in K[x]$ is isotrivial if $f$ is defined over $\Qbar$ 
up to a change of variables, that is, if $\varphi^{-1}\circ f\circ \varphi\in \Qbar[x]$ for 
some $\varphi\in \Kbar[x]$ of degree $1$. In the special case of a unicritical polynomial $f(x)=x^d+c\in K[x]$, we have that $f$ is isotrivial if and only if $c\in\Qbar$. We say $\beta\in \Kbar$ is 
periodic for $f$ if $f^n(\beta)=\beta$ for some $n\geq 1$, and we say $\beta$ is preperiodic for $f$ if $f^m(\beta)$ is periodic for some $m \ge 0$. Finally, we say that $\beta$ is postcritical for $f$ if $f^n(\alpha)=\beta$ for some $n\geq 1$ and some critical point $\alpha$ of $f$.

With this notation, our first main theorem is as follows. 

\begin{thm}\label{p-theorem}
  Let $q = p^{r}$ ($r\ge 1$) be a power of the prime number $p$, let $c\in K\setminus \Qbar$, let
  $f(x) = x^q + c \in K[x]$ and let $\beta\in K$. Then the 
  following are equivalent:
  \begin{enumerate}
  \item The point $\beta$ is neither periodic nor postcritical for $f$.
  \item The group $G_\infty(\beta)$ has finite index in $[C_q]^\infty$.
  \end{enumerate}
\end{thm}

All the methods used in the proof of Theorem~\ref{p-theorem} work for unicritical polynomials of any 
degree  $d$, except that we need the degree to be a prime power for proving the eventual stability of $(f, \beta)$ 
(see Theorem~\ref{stable-theorem} below and Section~\ref{stable}). 
In the  case where $q=2$, this means that  $G_\infty(\beta)$ has
finite index in $\Aut(T^2_\infty)$. For larger $q$ this index is infinite, as mentioned previously. The case 
of isotrivial polynomials (i.e., when $c\in\Qbar$ in Theorem~\ref{p-theorem}) is very different and will be dealt with in Section~\ref{isotrivial case}.

It is fairly easy to see that the conditions 
on $\beta$ in Theorem~\ref{p-theorem} are necessary. If $\beta$ is periodic or postcritical, 
then $[[C_q]^\infty:G_\infty(\beta)]=\infty$ by a straightforward argument (see Proposition~\ref{necessary conditions}). 
Most of the paper is devoted to the showing that these conditions 
are sufficient. 

\begin{rem}
In general one needs to rule out postcritically finite (PCF) maps in order to obtain a finite index result, as in 
the main result of~\cite{BT2}. The reason we do not need to do this in Theorem~\ref{p-theorem} is 
that a PCF polynomial of the form $f(x)=x^q+c$ is automatically isotrivial. This is because 
$c$ satisfies an equation of the form $f^n(c)=f^m(c)$ for some $n>m\geq 0$, and so $c\in \Qbar$. For isotrivial 
polynomials the PCF distinction regains its importance; see Section~\ref{isotrivial case}.
\end{rem}

One of the key steps in the proof of Theorem~\ref{p-theorem} 
is an eventual stability result. As is usual in arithmetic dynamics, we say that the pair $(f,\beta)$ is eventually stable over the field $K$ if 
the number of irreducible $K$-factors of $f^n(x)-\beta$ is uniformly bounded for all $n$.

\begin{thm}\label{stable-theorem}
  Let $q = p^{r}$  ($r\ge 1$) be a power of the prime number $p$.  Let
  $f \in K[x]$ be a polynomial of the form $x^q + t$ where $t\notin \Qbar$. 
  Then for any non-periodic $\beta \in K$, the
  pair $(f,\beta)$ is eventually stable over $K$.  
\end{thm}

We also prove the following disjointness theorem for fields generated
by inverse images of different points under different maps.  

\begin{thm}\label{disjoint-theorem} 
  For  $i =1, \dots, n$ let $f_i(x) = x^q + c_i \in K[x]$, where
  $c_i\notin \Qbar$, and let $\alpha_i\in K$.  Suppose that there are no
  distinct $i, j$ with the property that $(\alpha_i, \alpha_j)$ lies on a curve in $\A^2$
  that is periodic under the action of $(x,y) \mapsto (f_i(x),
  f_j(y))$.  For each $i$, let $M_i$ denote 
  $K_\infty(f_i,\alpha_i)$.  Then for each $i=1,\dots, n$, we have that 
\[ \left[M_i \cap \left(\prod_{j \ne i} M_j\right): K\right] < \infty .\]
\end{thm}

Theorem \ref{disjoint-theorem} also has a natural interpretation as a finite
index result across pre-image trees of several points (see Section \ref{multitree}).

\begin{rem}
In light of Odoni's work, unicritical polynomials with degree $d\geq 3$ cannot be considered generic 
from the point of view of arboreal Galois theory (indeed, they are not a generic family in the 
moduli space of degree $d$ polynomials in any reasonable sense). There are other families of polynomials 
and rational functions (such as postcritically finite maps) that arise 
as obstructions to any potential classification of finite index arboreal representations  
-- see~\cite[Section 3]{RafeArborealSurvey} and~\cite[Prop 3.3]{BT2} for examples. One might hope that in 
these ``exceptional" families, something similar to Theorem~\ref{p-theorem} could hold, in that a broad finite 
index result could be established for a natural overgroup other than $\Aut(T^d_\infty)$. 
The authors will explore this in future work.
\end{rem}

\medskip

{\bf Acknowledgments.} D.G. was partially supported by an NSERC Discovery grant. L.-C. H. was partially supported by 
MOST Grant 106-2115-M-003-014-MY2. 

\section{Wreath products}\label{wreath}

In this section we give a brief introduction to wreath products, which arise naturally  
from the Galois theory of the preimage fields $K_n(\beta)=K(f^{-n}(\beta))$.

Let $G$ be a permutation group acting on a set $X$, 
and let $H$ be any group. Let $H^X$ be the group of functions from $X$ to $H$ 
with multiplication defined pointwise, or equivalently the direct product of $|X|$ copies of $H$. 
The wreath product of $G$ by $H$ is the semidirect 
product $H^X\rtimes G$, where $G$ acts on $H^X$ 
by permuting coordinates: for $f\in H^X$ and $g\in G$ we have \[f^g(x)=f(g^{-1}x)\]
for each $x\in X$. We will use the notation $G[H]$ for the wreath product, suppressing the set $X$ 
in the notation. (Another common convention is $H\wr G$ or $H\wr_X G$ if we wish to call attention to $X$.)

Fix an integer $d\geq 2$. For $n\geq 1$, let $T^d_n$ be the complete rooted $d$-ary tree of level $n$. 
It is easy to see that $\Aut(T^d_1)\cong S_d$, and standard to show that $\Aut(T^d_n)$ 
satisfies the recursive formula
\[\Aut(T^d_n)\cong \Aut(T^d_{n-1})[S_d].\]
Therefore we may think of $\Aut(T^d_n)$ as the ``$n$th iterated wreath product" of $S_d$, which we will 
denote $[S_d]^n$. In general, for $f\in K[x]$ of degree $d$ and $\beta\in K$, the Galois group $G_n(\beta)=\Gal(K_n(\beta)/K)$ 
embeds into $[S_d]^n$ via the faithful action of $G_n(\beta)$ on the $n$th level of the tree of preimages of $\beta$ 
(see for example~\cite{OdoniIterates} or~\cite[Section 2]{BT2}).

Assume now that $f(x):=x^d+c\in K[x]$, where $K$ is a field of characteristic $0$ that contains the $d$th roots of unity. 
For $\beta\in K$ such that $\beta-c$ is not a $d$th power in $K$, 
we have $K_1(\beta)=K((\beta-c)^{1/d})$ and $G_1(\beta)\cong C_d$. For any $n\geq 2$, 
the extension $K_{n}(\beta)$ is a Kummer extension attained by adjoining to $K_{n-1}(\beta)$ the 
$d$th roots of $z-c$ where $z$ ranges over the roots of $f^{n-1}(x)=\beta$. Thus we have
\begin{equation*}
\Gal(K_n(\beta)/K_{n-1}(\beta))\subseteq \prod_{f^{n-1}(z)=\beta} \Gal(K_{n-1}(\beta)((z-c)^{1/d})/K_{n-1}(\beta))\subseteq C_d^{d^{n-1}}.
\end{equation*}
This is clear if $f^{n-1}(x)-\beta$ has distinct roots in $\Kbar$. 
If $f^{n-1}(x)-\beta$ has repeated roots, then $\Gal(K_n(\beta)/K_{n-1}(\beta))$ sits inside a direct product of a 
smaller number of copies of $C_d$, so the stated containments still hold. 

Considering the Galois tower
\[ K_n(\beta) \supseteq K_{n-1}(\beta)\supseteq K\]
we see that 
\[G_n(\beta)\subseteq \Gal(K_n(\beta)/K_{n-1}(\beta)) \rtimes G_{n-1}(\beta) \cong G_{n-1}(\beta)[C_d],\]
where the implied permutation action of $G_{n-1}(\beta)$ is on the set of roots of $f^{n-1}(x)-\beta$. By induction,
$G_n(\beta)$ embeds into $[C_d]^n$, the $n$th iterated wreath product of $C_d$. 
Observe that $[C_d]^n$ sits as a subgroup of $\Aut(T^d_n)\cong[S_d]^n$ via the obvious action on the tree. 
Taking inverse limits, $G_\infty(\beta)$ embeds into $[C_d]^\infty$, which sits as a subgroup of $\Aut(T_\infty)$. 

We summarize our basic strategy for proving that $G_\infty(\beta)$ has finite or infinite index in $[C_d]^\infty$ 
as Proposition~\ref{indexprop}.

\begin{prop}\label{indexprop}
Let $f=x^d+c\in K[x]$. Then $[[C_d]^\infty:G_\infty(\beta)]<\infty$ if and only if $\Gal(K_n(\beta)/K_{n-1}(\beta))\cong C_d^{d^{n-1}}$ 
for all sufficiently large $n$.
\end{prop}
\begin{proof}
Consider the projection map $\pi_n:[C_d]^\infty\to[C_d]^n$. The restriction of $\pi_n$ maps $G_\infty(\beta)$ to 
$G_n(\beta)$. By basic group theory,
\[
[[C_d]^\infty:G_\infty(\beta)]\geq [[C_d]^n:G_n(\beta)].
\]
Therefore if $\Gal(K_n(\beta)/K_{n-1}(\beta))$ is a proper subgroup of $C_d^{d^{n-1}}$ for infinitely many $n$, then 
$[[C_d]^n:G_n(\beta)]$ is unbounded as $n\to\infty$, and $[[C_d]^\infty:G_\infty(\beta)]=\infty$.

Conversely, by appealing to the profinite structure of $[C_d]^\infty$ we see that distinct cosets of $G_\infty(\beta)$ in $[C_d]^\infty$ must 
project to distinct cosets in $[C_d]^n$ under $\pi_n$ for some $n$. If there exists $N$ such that $\Gal(K_n(\beta)/K_{n-1}(\beta))\cong C_d^{d^{n-1}}$ for all $n > N$, then by induction, 
\[[[C_d]^n:G_n(\beta)]\leq [[C_d]^N:G_N(\beta)]\] for all $n$. Thus $[[C_d]^\infty:G_\infty(\beta)]\leq [[C_d]^N:G_N(\beta)]$ as well.
\end{proof}

\section{Necessary conditions}

We prove that the conditions in Theorem~\ref{p-theorem} are necessary for finite index. 
For this part of the theorem, we do not need to assume that $f$ has prime power degree, or that $f$ is not isotrivial.
The argument relies on a basic fact of algebra known as Capelli's Lemma, which we will use many times throughout the paper. 
We state it below without proof.

\begin{lem}[Capelli's Lemma]
Let $K$ be any field and let $f,g\in K[x]$. Suppose $\alpha\in\Kbar$ is any root of $f$. 
Then $f(g(x))$ is irreducible over $K$ if and only if both $f(x)$ is irreducible over $K$ and $g(x)-\alpha$ 
is irreducible over $K(\alpha)$.
\end{lem}

\begin{prop}\label{necessary conditions}
  Suppose $f(x) = x^d + c \in K[x]$ with $d\geq 2$, and let $\beta\in K$. If  $\beta$ is either 
  periodic or postcritical for $f$, then $[[C_d]^\infty:G_\infty(\beta)]=\infty.$
\end{prop}
\begin{proof}
First assume that $\beta$ is postcritical for $f$, i.e., that there is some critical point $\alpha$ of $f$ 
with $f^m(\alpha)=\beta$ for some $m\geq 1$. This means that the tree of preimages of $\beta$ is degenerate 
at the $m$th level: 
as $f^m(x)-\beta$ has at least one repeated root, we have \[|f^{-n}(\beta)|\leq (d^m-1)d^{n-m}\] for every $n\geq m$. 
As in Section~\ref{wreath}, the Galois group $\Gal(K_n(\beta)/K_{n-1}(\beta))$ embeds into the 
direct product of $|f^{-(n-1)}(\beta)|$ copies of $C_d$. In particular, 
\[|\Gal(K_n(\beta)/K_{n-1}(\beta))|\leq d^{(d^m-1)d^{n-1-m}}<d^{d^{n-1}}\]
for all sufficiently large $n$. By Proposition~\ref{indexprop}, we conclude that $G_\infty(\beta)$ has infinite 
index in $[C_d]^\infty$.

Now assume that $\beta$ is periodic for $f$ and \emph{not} postcritical, so that the tree of preimages of
$\beta$ can be identified with the complete $d$-ary tree $T^d_\infty$. The pair $(f,\beta)$ cannot be eventually stable 
by a straightforward argument using Capelli's Lemma~\cite[Prop 4.2]{BT2}. This implies 
that the number of Galois orbits in $f^{-n}(\beta)$ is unbounded as $n\to\infty$, and 
thus that there are an infinite number of orbits in the action of $G_\infty(\beta)$ on $\partial T^d_\infty$, 
where $\partial T_\infty^d$ is the boundary (or the ``ends") of the tree $T_\infty^d$, which can be identified 
with the set of infinite paths starting from the root of the tree~\cite[Prop 2.2]{RafeAlon}. 
But as $[C_d]^n$ acts transitively on the $n$th level of the tree for every $n$, 
we see that $[C_d]^\infty$ acts transitively on $\partial T^d_\infty$. A simple argument in group theory then 
implies that $[[C_d]^\infty:G_\infty(\beta)]=\infty$~\cite[Prop 3.3]{RafeAlon}. 
\end{proof}

\section{Height Estimates}\label{heights}
In this section we present two lemmas that give key height inequalities, which will 
be used in the proofs of the main theorems. For background on heights, see 
~\cite{HindrySilverman,GNT,BridyTucker}.

First we set some notation. We continue with the assumption that $K$ is a function field of transcendence degree $1$ over $\Qbar$. 
Choose a place $\fq$ of $K$ and set 
\[
\fo_K=\{z\in K:v_\fp(z)\geq 0\text{ for all }\fp\neq\fq\}.
\]
Let $\fp$ be a non-archimedean prime of $K$, which gives a prime of $\fo_K$. 
Let $k_\fp$ be the residue field $\fo_K/\fp$; note that $k_\fp$ is naturally isomorphic to $\Qbar$. Then for each point $z\in K$, we have its Weil height 
$$h(z):=\sum_{\fp\text{ place of }K} -\min\{0,v_\fp(z)\}.$$
For $f\in K[x]$ with $\deg f=d\geq 2$, 
let $h_f(z)$ be the Call-Silverman canonical height of $z$ relative to $f$~\cite{CallSilverman}, defined by
\[
h_f(z) = \lim_{n\to\infty}\frac{h(f^n(z))}{d^n}.
\]
We will often write sums indexed by primes of $\fo_K$ that satisfy some condition.
As an example of our indexing convention, observe that
\[
\sum_{v_\fp(z)>0} v_\fp(z)\leq h(z)
\]
for all $z\in K^\times$ by the product formula for $K$. Also define the forward orbit of $\gamma\in K$ under $f\in K[x]$ to be 
\[
\cO_f(\gamma)=\bigcup_{n\geq 1} \left\{f^n(\gamma)\right\}=\left\{f(\gamma),f^2(\gamma),f^3(\gamma),\dots\right\}.
\]
With this notation we have the following two lemmas.

\begin{lem}\label{from-5.1}
Let $f\in K[x]$ with $d=\deg(f)\geq 2$. Let $\gamma\in K$ with $h_f(\gamma)>0$.  Let
  $\beta_1, \beta_2 \in K$ such that $\beta_2\notin\O_f(\beta_1)$. 
  For $n>0$, let $\cX(n)$ denote the
  set of primes $\p$ of $\fo_K$ such that
\[ 
\min(v_\p(f^m(\gamma)-\beta_1), v_\fp(f^n(\gamma) - \beta_2)) > 0
\]
for some $0 < m < n$. Then for any $\delta >0$, we have
\begin{equation*}
\#\cX(n) \leq \delta d^n h_f(\gamma)+ O_\delta(1).  
\end{equation*}
for all $n$. (Note that, with the notation as in \cite{BT2}, we have that $N_\fp=1$ for each place $\fp$ since $k_\fp$ is isomorphic to the field of constants.)
\end{lem}
\begin{proof}
See~\cite[Section 5]{BT2}. Note that $\beta_1$ and $\beta_2$ need not be distinct.
\end{proof}

\begin{lem}\label{from-Roth}
  Let $f\in K[x]$ with $d=\deg(f)\geq 2$, and assume that $f$ is not isotrivial. 
  Let $\gamma,\beta\in K$ be such that $\beta\notin\O_f(\gamma)$ and that
  $\beta$ is also not postcritical.  
  For every $\epsilon > 0$, there exists a constant
  $C_\epsilon$ such that
\[
\#\left\{\fp\colon v_\p(f^n(\gamma)-\beta)= 1\right\}\geq
(d-\epsilon)d^{n-1}h_f(\gamma) + C_\epsilon
\]
for all $n\geq 1$.
\end{lem}
\begin{proof}
This follows immediately from \cite[Lemma 4.2]{BridyTucker}.  
\end{proof}

Lemma~\ref{from-Roth} is sometimes called the ``Roth-$abc$" estimate because of its similarity to 
Roth's theorem; for our case of function fields, this is a consequence of Yamanoi's proof of  
Vojta's $(1+\epsilon)$-conjecture~\cite{Yamanoi}.  

\section{Finiteness of GCD}\label{function-case}

We derive the following theorem by combining the results from \cite{UnicriticalAndreOort} and \cite{CallSilverman}.

\begin{prop}\label{FunctionFinite}
Let $K$ be the function field of a smooth, projective, irreducible curve $X$ defined over $\Qbar$. For $i = 1, 2$, let $f_i(x) = x^d + t_i \in K[x]$ be non-isotrivial
  with $t_1 \not= \xi t_2$ for any $\xi$ such that $\xi^{d-1} = 1$.  Let
  $c_1, c_2 \in K$ be such that $f_i^\ell(0) \ne c_i$ for $i = 1, 2$
  and for all non-negative integers $\ell$.  Then there are at most
  finitely many places $v$ of $K$ such that there are positive
  integers $m,n$ with the property that
\begin{equation}
\label{eq:finite GCD}
\min(v(f_1^m(0) - c_1), v(f_2^n(0) - c_2)) > 0. 
\end{equation}
\end{prop}

In order to prove Proposition~\ref{FunctionFinite}, we need a Bogomolov-type version of the main theorem of \cite{UnicriticalAndreOort}.  
\begin{thm}\label{AOQ}
Let $d\in\N$, let $F$ be a number field, let $K$ be the function field of a smooth, projective, geometrically irreducible curve $X$ defined over $F$, and let $t_1,t_2\in K\setminus F$ such that $t_1/t_2$ is not a $(d-1)$-st root of unity. We let $f_i(x):=x^d+t_i\in K[x]$ for $i=1,2$, and 
for each point $\lambda\in X(\Qbar)$ which is not a pole for either $t_1$ or $t_2$,  we consider the specialization of the polynomials $f_i$ at $\lambda$, denoted as $f_{i,\lambda}(x):=x^d+t_i(\lambda)\in \Qbar[x]$; for each such $\lambda\in X(\Qbar)$, we denote by $h_{f_{i,\lambda}}:\Qbar\lra \R_{\ge 0}$  the corresponding canonical heights. Then there exists $\epsilon>0$ such that there are finitely many points $\lambda\in X(\Qbar)$ for which $\max\{h_{f_{1,\lambda}}(0),h_{f_{2,\lambda}}(0)\}\le \epsilon$.    
\end{thm}

\begin{proof}
We argue by contradiction and therefore assume there exists an infinite sequence of points $\{\lambda_i\}_{i\ge 1}\subseteq X(\Qbar)$ such that 
\begin{equation}
\label{eq:tends to 0}
\lim_{i\to\infty} h_{f_{1,\lambda_i}}(0)=\lim_{i\to\infty} h_{f_{2,\lambda_i}}(0)=0.
\end{equation}
We proceed as in \cite{UnicriticalAndreOort} and for each $j=1,2$, we construct adelic metrized line bundles $\overline{\cL_j}$ on the curve $X$ corresponding to the families of dynamical systems $z\mapsto z^d+t_1(\lambda)$, respectively $z\mapsto z^d+t_2(\lambda)$ (parametrized by the $\Qbar$-points $\lambda\in X$). In general, given a rational function $\psi:X\lra \bP^1$ (defined over $\Qbar$), there exists an adelic metrized line bundle $\overline{\cL_\Psi}$ associated to the family of dynamical systems $g^{\Psi}_\lambda(z):=z^d+\Psi(\lambda)$ (as we vary $\lambda\in X(\Qbar)$), where $\cL_\Psi$ is the line bundle on $X$ obtained by pulling-back $\mathcal{O}(1)$ through the morphism $\Psi:X\lra \bP^1$; for more details, see \cite[Sections~3.2~and~4.1]{UnicriticalAndreOort}. In particular, this gives rise to height functions $h_{\overline{\cL_j}}:X(\Qbar)\lra \R_{\ge 0}$ (associated to the metrized line bundles $\overline{\cL_j}$, for $j=1,2$) for which we have:
\begin{equation}
\label{eq:2 heights}
h_{\overline{\cL_j}}(\lambda)= \frac{d}{\deg(t_j)}\cdot h_{f_{j,\lambda}}(0),
\end{equation}
for each $\lambda\in X(\Qbar)$, where $\deg(t_j)$ is the degree of the rational function $t_j:X\lra \bP^1$; see \cite[Proposition~3.5]{UnicriticalAndreOort}. Our hypothesis (see \eqref{eq:tends to 0}), coupled with \eqref{eq:2 heights},  yields that for the infinite sequence $\{\lambda_i\}_{i\ge 1}\subseteq X(\Qbar)$, we have
\begin{equation}
\label{eq:tends to 0 2}
\lim_{i\to\infty} h_{\overline{\cL_1}}(\lambda_i)=\lim_{i\to\infty} h_{\overline{\cL_2}}(\lambda_i)=0.
\end{equation}
Using \eqref{eq:tends to 0 2} coupled with \cite[Proposition~3.4.2]{CL-2} (which uses crucially the inequalities established by Zhang \cite{Zhang-2} regarding the successive minima associated to a metrized line bundle), we derive that there exist positive integers $\ell_j$ (for $j=1,2$) such that the two line bundles $\cL_1^{\ell_1}$ and $\cL_2^{\ell_2}$ are linearly equivalent and, moreover, the two heights $h_{\overline{\cL_j}}$ are proportional. In particular, this means that for each $\lambda\in X(\Qbar)$, we have that $h_{\overline{\cL_1}}(\lambda) = 0$ if and only if $h_{\overline{\cL_2}}(\lambda)=0$. Using this last equivalence along with equation \eqref{eq:2 heights}, we obtain that for each point $\lambda\in X(\Qbar)$,
\begin{equation}
\label{eq:iff}
h_{f_{1,\lambda}}(0) = 0 \text{ if and only if }h_{f_{2,\lambda}}(0)=0.
\end{equation}
Using \eqref{eq:iff} and the fact that only preperiodic points have canonical height equal to $0$ (for a rational function defined over $\Qbar$), we get that $0$ is preperiodic for the dynamical system $z\mapsto z^d+t_1(\lambda)$ if and only if $0$ is preperiodic for the dynamical system $z\mapsto z^d+t_2(\lambda)$. In other words, $\gamma_1:=t_1(\lambda)\in \Qbar$ is a PCF (postcritically-finite) parameter for the family of unicritical polynomials  $z\mapsto z^d+\gamma$ (parametrized by $\gamma\in\Qbar$) if and only if $\gamma_2:=t_2(\lambda)$ is a PCF parameter for the same dynamical system $z\mapsto z^d+\gamma$. Because there exist infinitely many PCF parameters $\gamma\in\Qbar$ for the family of polynomials $z\mapsto z^d+\gamma$, we conclude that there exist infinitely many $\lambda\in X(\Qbar)$ such that 
\begin{equation}
\label{eq:infinitely PCF points}
t_1(\lambda)\text{ and }t_2(\lambda)\text{ are PCF parameters.}
\end{equation}
Now, we let $Y$ be the Zariski closure in the plane of the image of $X$ under the rational map $X\dashrightarrow \A^2$ given by $x\mapsto (t_1(x), t_2(x))$; then \eqref{eq:infinitely PCF points} yields that there exist infinitely many points $(\gamma_1,\gamma_2)\in Y(\Qbar)$ with both coordinates PCF parameters for the unicritical dynamical system $z\mapsto z^d+\gamma$. But then \cite[Theorem~1.1]{UnicriticalAndreOort} yields that $Y\subset \A^2$ is given by an equation of the form $y=\zeta\cdot x$ for some $(d-1)$-st root of unity $\zeta$, where $(x,y)$ are the coordinates of $\A^2$ (note that $Y$ is neither a horizontal line, nor  a vertical line because both $t_1$ and $t_2$ are non-constant rational functions, and so possibilities (1)-(2) in \cite[Theorem~1.1]{UnicriticalAndreOort} cannot occur). However, our hypothesis regarding $t_1/t_2$ not being a $(d-1)$-st root of unity prevents $Y$ from satisfying such an equation and this contradiction proves that there exists no infinite sequence $\{\lambda_i\}_{i\ge 1}\subseteq X(\Qbar)$ as in \eqref{eq:tends to 0}. This concludes our proof of Theorem~\ref{AOQ}. 
\end{proof}

We now state a simple lemma that follows from work of Call and
Silverman \cite{CallSilverman}.  We recall the following lemma from
\cite[Lemma~8.3]{BT2}.  

\begin{lem}\label{FromCS}
Let $K$ be the function field of a smooth, irreducible curve defined over $\Qbar$. For each $t\in K$ and each $\lambda\in X(\Qbar)$ which is not a pole of $t$, we denote by $t_\lambda:=t(\lambda)$ the specialization of $t$ at $\lambda$.  Let $\varphi \in K(x)$ have degree $d \geq 2$.  Let $y, z \in K$ with
  $h_\varphi(y) > 0$ (where $h_\varphi$ is the canonical height associated to the rational function $\varphi$).  Let $(\lambda_i)_{i=1}^\infty$ be a sequence of points of 
  $X(\Qbar)$ satisfying
  $\varphi_{\lambda_i}^{n_i}(y_{\lambda_i}) = z_{\lambda_i}$ for a
  sequence $(n_i)_{i=1}^\infty$ of positive integers with
  $\lim_{i\to\infty} n_i = \infty$. Then
\begin{equation*}\label{to-zero}
\lim_{i \to \infty} h_{\varphi_{\lambda_i}}(y_{\lambda_i}) = 0 .
\end{equation*}
\end{lem}

\begin{proof}[Proof of Proposition~\ref{FunctionFinite}.]
Let $F$ be a number field such that $X$ is a geometrically irreducible curve defined over $F$. 
If there were infinitely many places $v$ of the function field $K$ such that  \eqref{eq:finite GCD} holds, then this means there exists an infinite sequence of points $\{\lambda_i\}_{i\ge 1}\subseteq X(\Qbar)$ such that there exist some nonnegative integers $m_i$ and $n_i$, for which we have
\begin{equation}
\label{eq:simultaneous intersection}
f_{1,\lambda_i}^{m_i}(0) = c_1(\lambda_i)\text{ and }f_{2,\lambda_i}^{n_i}(0)=c_2(\lambda_i).
\end{equation}
Also, since $f_1^m(0)\ne c_1$ for all integers $m\ge 0$ and $f_2^n(0)\ne c_2$ for all integers $n\ge 0$, we derive that the integers $m_i$ and $n_i$ appearing in \eqref{eq:simultaneous intersection} must tend to infinity. But then Lemma~\ref{FromCS} yields that 
$$\lim_{i\to\infty} h_{f_{1,\lambda_i}}(0) = \lim_{i\to\infty} h_{f_{2,\lambda_i}}(0) = 0,$$
contradicting thus Theorem~\ref{AOQ}.  
\end{proof}

\section{Eventual Stability} \label{stable}

The results of this section are valid (with only a few changes) in the more general setting of any function field  of a curve defined over a finitely generated field  of characteristic $0$. However, we will restrict to the case relevant for our results. So, let $L$ be a number field, let $k$ be the function field of a smooth, projective, geometrically irreducible curve $C$ defined over $L$,  let $q = p^{r} $ be a power of a prime number $p$ and let $f(x):=x^q+t\in k[x]$ for some $t\in k\setminus L$. (Note that our hypothesis yields that $L$ is algebraically closed in $k$.) Let $\beta\in k$ be a point which is not periodic under $f$;  
then we will prove that the pair $(f,\beta)$ is eventually stable over $k$.

We note that the places of $k$ correspond to points of 
$C(\Qbar)$. For any element
$c \in k$ and any point $\lambda$ of $C(\Qbar)$ such that $c$ does
not have a pole at $\lambda$, we let $c_\lambda$ denote the
specialization of $c$ to $\Qbar$ at $\lambda$ (see
\cite{CallSilverman} for more details); in other words, seeing $c$ as a rational function $C\lra \bP^1$ defined over $L$, then $c_\lambda:=c(\lambda)$.  Likewise for a rational
function $\varphi \in k(x)$, we let $\varphi_\lambda$ denote the
specialization of $\varphi$ to $\Qbar(x)$ at $\lambda$ for any
$\lambda$ such that the coefficients of $\varphi$ do not have poles at
$\lambda$; so, for our polynomial $f(x)=x^q+t$, we simply have that $f_\lambda(x):=x^q+t_\lambda$ for any point $\lambda\in C(\Qbar)$ which is not a pole of $t\in k$.  Finally, for any $\lambda\in C(\Qbar)$, we let $L(\lambda)$ be the field of definition for the point $\lambda$; in particular, this means that $[L(\lambda):L]<\infty$ and furthermore, for each $z\in k$, we have that $z_\lambda\in L(\lambda)$.

With notation as above, the following lemma follows from the work of
\cite{CallSilverman}.

\begin{lem}\label{fromCS}
  Let $\varphi \in k(x)$ be a non-isotrivial rational function of degree greater than one and let
  $\beta \in k$.  Then we have the following:
  \begin{enumerate}
  \item[(a)] if $h_\varphi(\beta) >  0$, then the set
  of specializations $\lambda$ from $k$ to $\Qbar$ such that
  $h_{\varphi_\lambda}(\beta_\lambda) = 0$ has bounded height (with respect to some degree-$1$ divisor on $C$); and
  \item[(b)] if $\beta$ is not periodic under $\varphi$, then  the set
  of specializations $\lambda$ from $k$ to $\Qbar$ such that
  $\beta_\lambda$ is periodic under $\varphi_\lambda$ has bounded height.
\end{enumerate}
\end{lem}
\begin{proof}
  The statement of (a) follows directly from \cite[Theorem
  4.1]{CallSilverman}. We now prove (b).   If $\beta$ is not preperiodic under $\varphi$,
  then $h_{\varphi}(\beta) > 0$, by \cite{Baker1}.  Then, from (a), it
  follows that the set of $\lambda$ such that
  $h_{\varphi_\lambda}(\beta_\lambda) = 0$ has bounded height.  Now,
  if $\beta$ is strictly preperiodic, then there are at most finitely
  many $\lambda$ such that $\beta_\lambda$ is periodic under
  $\varphi_\lambda$, so the set of such $\lambda$ clearly has bounded
  height.
 \end{proof}

The following lemma is a simple consequence of the main theorem of
\cite{RafeAlon}.

\begin{lem}\label{JL}
Let $g(x) = x^q + c$ where $c$ is a element of a number field $L$ with
the property that $|c|_v \leq 1$ for some non-archimedean place $v$ of
$L$ such that $v | p$.  Then for any $\beta \in L$ that is not
periodic under $g$, the pair $(g,
\beta)$ is eventually stable over $L$.
\end{lem}
\begin{proof}
Let $k_v$ be the residue field at $v$.  Then reducing $g$ at $v$
induces a map $g_v: \bP^1(k_v) \lra \bP^1(k_v)$ such that every point
$\bP^1(k_v)$ has exactly one inverse image under $g_v$.  Let $\beta_v
\in \bP^1(k_v)$ denote the reduction of $\beta$ at $v$.  Then there is
an $n$ such that $g_v^{-n}(\beta_v)  = \{\beta_v\}$.  Theorem 1.7 of
\cite{RafeAlon} states that $(g, \beta)$ must therefore be eventually
stable over $L$.
\end{proof}

Now we can state the main result of this section, which is instrumental in proving Theorem~\ref{stable-theorem}.
\begin{prop}\label{stable1}
Let $k$ be the function field of a smooth, projective, geometrically irreducible curve $C$ defined over a number field $L$.  Let $f(x) = x^q + t \in k[x]$, where $t \notin
L$.  Then for any $\beta\in k$ that is not periodic under $f$, the pair
$(f, \beta)$ is eventually stable over $k$.
\end{prop}

\begin{proof}
We may choose a specialization $\lambda$ of $k$ to $\Qbar$ such that:
\begin{itemize}
\item[(i)]  $t_\lambda$ is an algebraic integer;  and 
\item[(ii)] $\beta_\lambda$ is not
  periodic under $f_\lambda$. 
\end{itemize}
Indeed,  for all but
  finitely many $z \in  \Qbar$, there is a $\lambda\in C(\Qbar)$ such that
  $t_\lambda = z$; furthermore, condition~(ii) is achieved for all points  $\lambda\in C(\Qbar)$ of sufficiently large height (by Lemma~\ref{FromCS}), while on the other hand, if the algebraic integer $t_\lambda$ has large height, then the point $\lambda$ must have large height (on $C$) as well.  Then, by  Lemma~\ref{JL}, the pair $(f_\lambda, \beta_\lambda)$ is eventually  stable over $L(\lambda)$, which implies that $(f, \beta)$ is  eventually stable over $k$.
\end{proof}

\section{Ramification and Galois theory}\label{Galois section}
Let $f(x) = x^q + t$ with $t \in K$.  In this section we define
\emph{Condition R} and \emph{Condition U} in terms of primes dividing
certain elements of $K$ related to the forward orbits of 0. In
Proposition~\ref{necessary} and~\ref{main Galois} we show that these
conditions control ramification in the extensions
$K(\beta) \subseteq K_n(\beta)$, with consequences for the Galois
theory of these extensions.  We begin with the following standard
lemma from Galois theory.

\begin{lem}\label{Galois}
Let $L_1,\dots,L_n$ and $M$ be fields all contained in some larger field. Assume that $L_1,\dots,L_n$ are finite extensions of $M$.
\begin{enumerate}
\item [(i)]  If $L_1,L_2$ are Galois over $M$ with $L_1\cap L_2=M$, then $L_1 L_2$ is Galois over $L_2$ and
$\Gal(L_1 L_2/L_2)\cong\Gal(L_1/M)$.
\item [(ii)] If $L_1,\dots,L_n$ are Galois over $M$ with $L_i\cap\prod_{j\neq i}L_j = M$ for each $i$, then $\Gal(\Pi_{i=1}^n L_i/M)\cong\prod_{i=1}^n\Gal(L_i/M)$.
\end{enumerate}
\end{lem}

Conditions R and U make use of the notion of good reduction 
of a map $f\in K(x)$ at a prime $\fp$. A polynomial
\[
f(x) = a_dx^d + a_{d-1}x^{d-1}+\dots + a_1x + a_0,
\]
has good reduction at $\fp$ if $v_\fp(a_d)=0$ and $v_\fp(a_i)\geq 0$ for $0\leq i\leq d-1$. See \cite{MortonSilverman} or \cite[Theorem 2.15]{SilvermanADS} for a more careful definition that also applies to rational functions. Clearly any $f$ has good reduction at all but finitely many $\fp$. The idea behind the definition is that if $f$ has good reduction at $\fp$, then $f$ commutes with the reduction mod $\fp$ map $\bar{\cdot}: \P^1(\Kbar)\to \P^1(\overline{k}_\fp)$. This is clear for polynomials (see \cite[Theorem 2.18]{SilvermanADS} for a proof for rational functions). We say that $f$ has good separable reduction at $\fp$ if the reduced map  $\bar{f}:\P^1(\overline{k}_\fp)\to\P^1(\overline{k}_\fp)$ is separable. 

\begin{defn}
  Let $\beta \in \Kbar$.  We say that a
  prime $\fp$ of $K(\beta)$ satisfies {\bf Condition R} at
  $\beta$ for $f$ and $n$ if the following hold:
\begin{enumerate}
\item[(a)] $f$ has good separable reduction at $\fp$;
\item[(b)]  $v_\fp(f^i(0) - \beta) = 0$ for all $0 \leq i < n$;
\item[(c)]  $v_\fp(f^n(0) - \beta) = 1$;
\item[(d)] $v_\fp(\beta) = 0$.
\end{enumerate}
\end{defn}

\begin{defn}
  Let $\beta \in \Kbar$.  We say
  that a prime $\fp$ of $K(\beta)$ satisfies {\bf Condition U} at
  $\beta$ for $f$ and $n$ if the following hold:
\begin{itemize}
\item[(a)] $f$ has good separable reduction at $\fp$;
\item[(b)] $v_\fp(f^i(0) - \beta) = 0$ for all $0 \leq i \leq n$;
\item[(c)] $v_\fp(\beta)=0$.
\end{itemize}
\end{defn}

\begin{prop}\label{necessary}
Let $\beta\in\Kbar$. Let $\p$ be a prime of $K(\beta)$ that satisfies
Condition U at $\beta$ for $f$ and $n$. 
Then $\fp$ is unramified in $K_n(\beta)$.
\end{prop}
\begin{proof}
This is the content of \cite[Proposition 3.1]{BridyTucker}. The proof in \cite{BridyTucker} is 
stated for $\beta\in K$, but works exactly the same if we allow $\beta\in\Kbar$ and replace $K$ with $K(\beta)$.
\end{proof}

\begin{prop}\label{main Galois}
Let $\beta\in\Kbar$. Suppose that $\fp$ is a prime of $K(\beta)$ that 
satisfies Condition R at $\beta$ for $n$ and that $f^n(x)
- \beta$ is irreducible over $K(\beta)$. Then 
\[ \Gal( K_{n}(\beta) / K_{n-1}(\beta)) \cong C_q^{q^{n-1}}.\]
Furthermore, $\fp$
does not ramify in $K_{n-1}(\beta)$ and does ramify in any field $E$ such that
$K_{n-1}(\beta) \subsetneq E \subset K_n(\beta)$.
\end{prop}
\begin{proof}
Observe that Condition R at $\beta$ for $n$ implies Condition U at $\beta$ for $n-1$. 
By Proposition~\ref{necessary}, $\fp$ does not ramify in $K_{n-1}(\beta)$.

Let $\bar{z}$ denote the image of $z\in \P^1(\Kbar)$ under the reduction mod $\fp$ map, which is well defined 
as long as $v_\fp(z)\geq 0$. Consider the map $\bar{f}:\P^1(\overline{k}_\fp)\to\P^1(\overline{k}_\fp)$ that comes from reducing $f$ at $\fp$, 
and recall that Condition R assumes that $f$ has good reduction at $\fp$. The unique critical point of $\bar{f}$ 
is $0\in k_\fp$. By (b) of Condition R, we see that $\bar{f}^{n-1}(x)-\bar{\beta}$ has no repeated roots. By (c) of 
Condition R, we see that $0\in\bar{f}^{-n}(\beta)$, and $0$ is totally ramified over $\bar{f}(0)=\bar{t}$ (in the sense 
of $\bar{f}$ as a morphism of $\P^1(\overline{k}_\fp)$).
So $\bar{f}^n(x)-\bar{\beta}$ has $0$ as a root of multiplicity $q$, and has no other repeated roots. So we may write
\begin{equation}\label{1}
\bar{f}^n(x)-\bar{\beta}=x^q h(x)
\end{equation}
where $h(0)\neq 0$ and $h$ has distinct roots.

Let $z_1,\dots,z_{q^{n-1}}$ be the roots of $f^{n-1}(x)-\beta$. We have the factorization
\begin{equation}\label{2}
f^n(x)-\beta = \prod_{i=1}^{q^{n-1}}\left(f(x)-z_i\right)
\end{equation}
in $K_{n-1}(\beta)[x]$. For each $i$, let $L_i$ denote the splitting field of $f(x)-z_i$ over $K(z_i)$, and 
let $M_i$ denote the splitting field of $f(x)-z_i$ over $K_{n-1}(\beta)$. Note that $M_i=K_{n-1}(\beta)\cdot L_i$. 
By Capelli's Lemma, for each $i$, $f(x)-z_i$ is irreducible over $K(z_i)$. Also note that the $z_i$ are all distinct mod $\fp$, 
because $\bar{f}^{n-1}(x)-\bar{\beta}$ has no repeated roots.

Let $\fq$ be a prime of $K_{n-1}(\beta)$ that lies over $\fp$. 
By~\eqref{1} and~\eqref{2}, there is precisely one $i\in\{1,\dots,q^{n-1}\}$ for which $f(x)-z_i$ has a repeated root mod $\fq$, 
and this root is repeated with multiplicity $q$. 
For this $i$, let $\fm$ be a prime of $M_i$ that lies over $\fq$. Then we have
\begin{equation}\label{3}
f(x)-z_i\equiv x^q\pmod{\fm}.
\end{equation}
By (c) of Condition R we have $v_\fq(f(0)-z_i)=v_\fp(f(0)-z_i)=1$, as $\fq$ is unramified over $\fp$. 
By~\eqref{3} we compute $v_\fm(f(0)-z_i)=q$. We conclude that $\fq$ is totally ramified in $M_i$ 
with $e(\fm/\fq)=q$, and for all $j\neq i$, $\fq$ does not ramify in $M_j$.

Now, $\fp$ is unramified in $K(z_i)$ (which is a subextension of $K_{n-1}(\beta)$), so $\fm\cap L_i$ ramifies over 
$\fm\cap K(z_i)$ with ramification index $q$. Using that $f(x)-z_i$ is irreducible over $K(z_i)$, $\Gal(L_i/K(z_i))$ contains 
the order $q$ inertia subgroup $I(\fm\cap L_i/\fm\cap K(z_i))$ and is therefore cyclic of order $q$ (as we know 
it is a subgroup of $C_q$). 
Taking the base change by $K_{n-1}(\beta)$, $\Gal(M_i/K_{n-1}(\beta))$ is a normal subgroup of $C_q$ that contains 
the order $q$ inertia subgroup $I(\fm/\fq)$, so $\Gal(M_i/K_{n-1}(\beta))\cong C_q$ as well.

As $f^n(x)-\beta$ is irreducible over $K(\beta)$, it follows easily that $f^{n-1}(x)-\beta$ is irreducible over $K(\beta)$ as well. 
Therefore all of the $z_i$ are Galois-conjugate. That is, for any $z_j\neq z_i$, there exists $\sigma\in G_{n-1}(\beta)$ such that 
$\sigma (z_i)=z_j$. Applying $\sigma$ to $\fq$, we obtain a prime $\sigma(\fq)$ of $K_{n-1}(\beta)$ that ramifies in $M_j$ with 
ramification index $q$ and does not ramify in $M_k$ for any $k\neq j$. Repeating the same argument as above, it follows that 
$\Gal(M_j/K_{n-1}(\beta))\cong C_q$. The disjointness of ramification shows that for each $j$ we have
\begin{equation}
M_j\cap \left(\prod_{k\neq j} M_j \right)= K_{n-1}(\beta).
\end{equation}
By Lemma~\ref{Galois}, we conclude that $\Gal(K_n(\beta)/K_{n-1}(\beta))\cong C_q^{q^{n-1}}$. To see the statement about intermediate fields, let $\fm_j$ be a prime of $M_j$ lying over $\sigma (\fq)$. Then the inertia group $I(\fm_j/\sigma(\fq))$ extends 
to an inertia group $I(\fm_j'/\sigma(\fq))$ for a prime $\fm_j'$ of $K_n(\beta)$ lying over $\fm_j$. The group $I(\fm_j'/\sigma(\fq))$ 
restricts to the identity on $M_k$ for $k\neq j$. Therefore $\Gal(K_n(\beta)/K_{n-1}(\beta))$ is generated by inertia groups of 
the form $I(\fm_j'/\sigma(\fq))$ for various $\sigma\in G_{n-1}(\beta)$. It follows that any intermediate field $E$ lying strictly between 
$K_{n-1}(\beta)$ and $K_n(\beta)$ is ramified over some $\sigma(\fq)$, and thus ramified over $\fp$.
\end{proof}

Before stating the last result from Galois theory we need, we need a
little notation. 

\begin{defn}
  Let $\ba = (\alpha_1, \dots, \alpha_s) \in \Kbar^{s}$,
  and   let $\bF = (f_1, \dots, f_s)$, where for each $i$, $f_i=x^q +
  t_i$ for some $t_i\in K\setminus\Qbar$.   We
  define $K_{n}(\bF,\ba)$ to be $ \prod_{i=1}^s K_n(f_i,\alpha_i)$.
\end{defn}

With this notation we have the following.

\begin{prop}\label{final Galois}
  Let $\ba = (\alpha_1, \dots, \alpha_s) \in L^{s}$ for $L$ a
  finite extension of $K$.  Suppose there exist primes
  $\fp_1,\dots,\fp_s$ of $L$ such that
\begin{itemize}
\item[(a)] $\fp_i \cap K(\alpha_i)$ satisfies Condition R at
$\alpha_i$ for $f_i$ and $n$;
\item[(b)] $\fp_i \cap K(\alpha_j)$ satisfies Condition U at
$\alpha_j$ for $f_j$ and $n$ for all $j \not= i$;
\item[(c)] $\fp_i \cap K(\alpha_i)$ does not ramify in $L$; and
\item[(d)] $f_i^n(x) - \alpha_i$ is irreducible over $K(\alpha_i)$ for
  $i = 1, \dots, s$.
\end{itemize}
Then  $\Gal(K_{n}(\bF,\ba)/ K_{n-1}(\bF,\ba)) \cong C_q^{s q^{n-1}}$.
Furthermore, there is no field $E$ with $K_{n-1}(\bF,\ba) \subsetneq E \subset
K_{n}(\bF,\ba)$ that is unramified over $K_{n-1}(\bF,\ba)$.
\end{prop}

\begin{proof}
Choose any $i$ with $1\leq i\leq s$. By Proposition~\ref{main Galois}, we have
\[\Gal(K_n(f_i,\alpha_i)/K_{n-1}(f_i,\alpha_i))\cong C_q^{q^{n-1}}.\]
Further, $\fp_i$ does not ramify in $K_{n-1}(f_i,\alpha_i)$ and does ramify in every field $E$ with 
$K_{n-1}(f_i,\alpha_i)\subsetneq E\subseteq K_n(f_i,\alpha_i)$. Let $L_i=K_n(f_i,\alpha_i)\cdot K_{n-1}(\bF,\ba)$. By (b), (c), and Proposition~\ref{necessary}, 
$\fp_i$ does not ramify in $K_{n-1}(\bF,\ba)$ or in $L_j$ for $j\neq i$. Therefore 
\[K_n(f_i,\alpha_i)\cap \prod_{j\neq i} L_j= K_{n-1}(f_i,\alpha_i)\] 
and \[K_n(f_i,\alpha_i)\cap K_{n-1}(\bF,\ba)=K_{n-1}(f_i,\alpha_i).\]
By Lemma~\ref{Galois} and the fact that $L_i\cdot\prod_{j\neq i} L_j=K_{n}(f_i,\alpha_i)\cdot \prod_{j\neq i} L_j$,
\begin{align*}
 \Gal\left(L_i\cdot \prod_{j\neq i}L_j /\prod_{j\neq i}L_j\right) & \cong\Gal(L_i/K_{n-1}(\bF,\ba))\\
 & \cong  \Gal(K_n(f_i,\alpha_i)/K_{n-1}(f_i,\alpha_i))\\
 & \cong C_q^{q^{n-1}}.
 \end{align*}
 Therefore we have
 \[ 
 q^{q^{n-1}}=\left[L_i:K_{n-1}(\bF,\ba)\right]\geq 
 \left[ L_i:L_i\cap  \prod_{j\neq i} L_j\right]\geq 
  \left[ L_i\cdot \prod_{j\neq i} L_j :  \prod_{j\neq i} L_j\right] =  q^{q^{n-1}}.
 \]
 We conclude that $L_i\cap \prod_{j\neq i} L_j=K_{n-1}(\bF,\ba)$. Applying Lemma~\ref{Galois} again 
 and using that $L_i\cdot \prod_{j\neq i} L_j = K_{n-1}(\bF,\ba)$, we are done.
\end{proof}


\section{Proof of Main Theorems}\label{main proof}

The proofs of the main theorems combine the preliminary arguments from 
throughout the paper with the following proposition, which uses height arguments to 
produce primes with certain ramification behavior in $K_n(\beta)$. 
Recall the definitions of Condition R and Condition U from Section \ref{Galois section}.
\begin{prop}\label{final prop}
  Let $K$ be the function field of a smooth, projective curve defined
  over $\Qbar$.  For $i = 1, \ldots, s$, let $f_i(x) =x^q + t_i \in K[x]$, where $t_i $ is not
  in the constant field of $K$.  Let $L$ be a finite extension of $K$.
  Let $\alpha_1, \dots, \alpha_s$ be distinct elements of $L$ such
  that for all $\ell$, the point $\alpha_\ell$ is not postcritical for
  $f_\ell$, and for any $i \not= j$, the point
  $(\alpha_i, \alpha_j) \in \A^2$ is not on a curve that is periodic
  under the action of $(x,y) \mapsto (f_i(x), f_j(y))$.  Then, for all
  sufficiently large $n$, there exist primes $\fp_1,\dots,\fp_s$ of
  $L$ such that
\begin{itemize}
\item[(a)] for each $i$, we have that $\fp_i \cap K(\alpha_i)$ satisfies Condition R at
$\alpha_i$ and $f_i$ for $n$;
\item[(b)] for each $i\ne j$, we have that $\fp_i \cap K(\alpha_j)$ satisfies Condition U at
$\alpha_j$ and $f_j$ for $n$;
\item[(c)] $\fp_i \cap K(\alpha_i)$ does not ramify in $L$.   
\end{itemize}
\end{prop}
\begin{proof}
For any $1\leq i\leq s$, let $\mathcal{A}_i(n)$ be the set of primes 
$\fp$ of $L$ such that (a), (b), and (c) hold. If a prime $\fp$ 
of $L$ satisfies condition R or condition U at $\alpha_i$ for $n$, then it is easy 
to see that the prime $\fp\cap K(\alpha_i)$ of $K(\alpha_i)$ also satisfies condition R or 
condition U at $\alpha_i$ for $n$. Therefore we will establish Conditions R and U for primes of $L$ rather than 
primes of the various $K(\alpha_i)$, which will make the argument less cumbersome to state. 
Thus all sums below are indexed by primes of $\fo_L$ as in Section~\ref{heights}.

There are only finitely many primes $\fp$ of $\fo_L$ for which some
$f_i$ does not have good separable reduction at $\p$,
$v_\fp(\alpha_i)\neq 0$ for some $i$, or $\fp\cap K(\alpha_i)$
ramifies in $L$ for some $\alpha_i$. The contributions of these primes
to our estimates will be absorbed into the constant term $C_\delta'$
at the end of the proof.

Now, note that after changing variables, we may assume that for any
$i, j$, we have $(t_i/t_j)^{q-1} = 1$ if and only if $t_i =
t_j$.  Choose any $i,j$ with $1\leq i,j\leq s$ such that $t_i\neq t_j$. 
By Proposition \ref{FunctionFinite}, there are at
most finitely many primes $\fp$ such that there are $m,n$ such that
\begin{equation} \label{diff-map}
  \min(v(f_i^m(0) - \alpha_i), v(f_j^n(0) - \alpha_j)) > 0.
\end{equation}
As above, the contributions of these primes
to our estimates will be absorbed into the constant term $C_\delta'$
at the end of the proof.

Now, take any $i$ and $j$ with $1\leq i,j\leq s$ (and possibly $i=j$)
such that $t_i = t_j$.  Note that when $t_i = t_j$, the curves
$(z, f_j^m(z))$ and $(f_i^m(z), z)$ are periodic under $(f_i, f_j)$ so
we may assume that there is no $m$ such that
$\alpha_i = f_j^m(\alpha_j)$ or $\alpha_j = f_i^m(\alpha_i)$ for any
$m \geq 0$, so we may use Lemma \ref{from-5.1}.  Let $\cX(n)$ be the set of primes $\fp$ with
$\min(v_\fp(f_i^n(0)-\alpha_j), v_\fp(f_i^m(0)-\alpha_i))>0$ for some
$1\leq m\leq n-1$. By Lemma \ref{from-5.1} with $\gamma=0$,
$\beta_1=\alpha_i$, and $\beta_2=\alpha_j$, for any $\delta > 0$ we have
\begin{equation}\label{0 earlier}
\# \cX(n) \leq \delta q^n h_{f_i}(0) + O_\delta(1),
\end{equation}
 Furthermore, if $j\neq i$, then the set of primes $\fp$ such that
 $\min(v_\fp(f_i^n(0)-\alpha_i),v_\fp(f_i^n(0)-\alpha_j))>0$ is a
 finite set depending only on $\alpha_i$ and $\alpha_j$ (and not on
 $n$), because $\alpha_i\equiv\alpha_j\pmod{\fp}$ for such $\fp$, so
\begin{equation}\label{0 same}
\#\left\{\fp\colon \min(v_\fp(f_i^n(0)-\alpha_i),v_\fp(f_i^n(0)-\alpha_j))>0\right\}  = O_\delta(1).
\end{equation}

Now, fix an $i$.  By Lemma \ref{from-Roth}, for any $\delta>0$ there
is a constant $C_\delta$ such that
\begin{equation}\label{starting point}
\# \left\{\fp\colon v_\fp(f_i^n(0)-\alpha_i)=1\right\} \geq (q-\delta)q^{n-1}h_{f_i}(0) + C_\delta.
\end{equation}
We subtract out \eqref{0 earlier}; for all such $j$ such that
$j\neq i$ we subtract \eqref{0 same} as well.  This sieves out all
$\fp\notin\mathcal{A}_i(n)$ from the sum in \eqref{starting point},
and we have
\begin{equation}\label{final equation}
\# \mathcal{A}_i(n)\geq q^n h_{f_i}(0)(1- s q
\delta - \delta/q) +C_\delta'
\end{equation}
where $C_\delta'$ is a constant obtained by combining all the $O_\delta(1)$ terms. For sufficiently large $n$, we can choose some $\delta$ to make the right hand side of \eqref{final equation} 
positive. Repeating this process for each $i$, we are done.
\end{proof}

\begin{proof}[Proof of Theorem \ref{stable-theorem}]
We know that $K$ is the function field of a projective, smooth, irreducible curve $C$ defined over $\Qbar$; so, we let $L$ be a number field such that $C$ is a geometrically irreducible curve defined over $L$. Then we let $k$ be the function field $L(C)$; at the expense of replacing $k$ by a finite extension, we may assume that $t,\beta\in k$.  
  By Proposition \ref{stable1}, the pair $(f,\beta)$ is eventually
  stable over $k$.   It will suffice to show that 
  $$\Qbar \cap \left(\bigcup_{n=1}^\infty k(f^{-n}(\beta))\right)\text{ is a finite extension of }L.$$ 
Since $(f,\beta)$ is eventually stable over $k$, by Capelli's Lemma there exists an
  $m$ such that $f^n(x) - \alpha_i$ is irreducible over $k(\alpha_i)$
  for all $\alpha_i$ such that $f^m(\alpha_i) = \beta$ and all $n\geq m$ 
  (see~\cite[Prop 4.2]{BT2} for a proof of this fact).  Applying
  Propositions \ref{final Galois} and \ref{final prop}, we see then that there is an integer $n_1$
  such that for all $n > n_1$, the field $k_n(\beta):=k(f^{-n}(\beta))$ contains no
  nontrivial extensions of $k_{n-1}(\beta)$ that are unramified
  over $k_{n-1}(\beta)$.  Let $\gamma$ be any element of
  $\Qbar \cap k_{\ell}(\beta)$ for some $\ell$.  Let $N$ be minimal among all
  integers such that $\gamma \in k_N(\beta)$.  Since
  $k_{N-1}(\beta)(\gamma)$ is unramified over
  $k_{N-1}(\beta)$, it follows that $N \leq n_1$.  Hence
  $\Qbar \cap \bigcup_{n=1}^\infty k(f^{-n}(\beta))$ is a finite extension of $L$, as
  desired. 
\end{proof}

\begin{proof}[Proof of Theorem \ref{p-theorem}]
 In Proposition~\ref{necessary conditions} we have already proved the conditions are necessary. Therefore assume 
 that $\beta$ is not postcritical nor periodic for $f$. 
 By Theorem \ref{stable-theorem}, the pair $(f,\beta)$ is eventually
 stable.  Again using Capelli's Lemma, there is some $m$ such that for all $\alpha\in f^{-m}(\beta)$ and for all $n\geq 1$, 
 $f^n(x)-\alpha$ is irreducible over $K_n(\beta)$. By Propositions~\ref{final Galois} and ~\ref{final prop}, 
 there exists $n_0$ such that for all $n\geq n_0$, we have
 \[\Gal(K_n(\beta)/K_{n-1}(\beta))\cong C_q^{q^{n-1}}.\]
 By Proposition~\ref{indexprop}, we are done.
\end{proof}

 \begin{proof} [Proof of Theorem \ref{disjoint-theorem}]
   By Theorem \ref{stable-theorem}, each pair $(f_i, \alpha_i)$ is
   eventually stable.  Arguing as in the proof of Theorem
   \ref{stable-theorem}, we let $\gamma$ in
   $M_i \cap \prod_{j \not= i} M_j$.  Let $N$ be minimal among all
   positive integers $n$ such that
   $\gamma \in K_n(f_i,\alpha_i) \cap \prod_{j \not= i} K_n(f_j,\alpha_j)$.  
   Applying Propositions \ref{final Galois} and \ref{final
     prop}, we see that there is an integer $n_2$ such for all $n > n_2$, we have
   \[ \Gal(K_{n}(\bF,\ba)/ K_{n-1}(\bF,\ba)) \cong C_q^{s q^{n-1}} \]
   (in the notation of Proposition \ref{final Galois}).  Since
   \[
[K_n(f_\ell,\alpha_\ell):  K_{n-1}(f_\ell,\alpha_\ell)]
  = [K_n(f_\ell,\alpha_\ell) \cdot K_{n-1}(\bF,\ba) : K_{n-1}(\bF,\ba)] = q^{q^{n-1}} 
  \] 
  for each $\ell$, it follows that any element of $K_n(f_i,\alpha_i) \cap \prod_{j \not= i}
K_n(f_j,\alpha_j)$ is in $K_{n-1}(\bF,\ba)$, and thus in both
$K_n(f_i,\alpha_i) \cap K_{n-1}(\bF,\ba) = K_{n-1}(f_i,\alpha_i)$ and
\[ \left(\prod_{j \not= i}
   K_n(f_j,\alpha_j) \right) \cap  K_{n-1}(\bF,\ba) = \prod_{j \not= i}
   K_{n-1}(f_j,\alpha_j) \]
Thus, we must have $N \leq n_2$ by minimality of $N$.   This shows that
$M_i \cap \prod_{j \not= i} M_j$ is a finite extension of $K$, as
desired.  
 \end{proof}

\section{The multitree}\label{multitree}

In this section we introduce a generalization of trees, which we call 
multitrees, in order to give a pleasant interpretation of Theorem~\ref{disjoint-theorem} in 
terms of a finite index statement. For our purposes, we can simplify the presentation of multitrees 
in~\cite[Section 11]{BT2} by avoiding the use of stunted trees.

Let $f\in K(x)$ with $\deg f\geq 2$ and set $\ba = \{\alpha_1,\dots,\alpha_s\}\subseteq K$. 
Define
\[
\cM_n(\ba)= \bigcup_{i=0}^n \bigcup_{j=1}^s f^{-i}(a_j)
\]
and
\[
G_n(\ba) = \Gal(K(\cM_n(\ba ))/K(\ba )).
\]
We refer to $\cM_n(\ba)$ as a \emph{multitree}. It can be pictured as the union of $s$ distinct trees of level $n$, rooted at the $\alpha_i$.

As $n\to\infty$, define the direct limit \[\cM_\infty(\ba) = \lim_{\longrightarrow} \cM_n(\ba)\]
 and the inverse limit \[G_\infty(\ba) = \lim_{\longleftarrow} G_n(\ba)\] 
 just as in the single tree case.
For each $n$, $G_n(\ba) $ acts faithfully on 
$\cM_n(\ba)$ in the usual way. So there are injections 
$G_n(\ba) \hookrightarrow\Aut(\cM_n(\ba))$, and thus an injection 
$G_\infty(\ba) \hookrightarrow\Aut(\cM_\infty(\ba))$, where an automorphism of the 
multitree must fix each root $\alpha_i$. 

Suppose that the individual trees rooted at $\alpha_i$ are disjoint, and that each $\alpha_i$ is neither periodic 
nor postcritical for $f$. Then the automorphism group of the infinite multitree has the simple description
\[\Aut(\cM_\infty(\ba))\cong \Aut(T_\infty^q)^s,\]
that is, the direct product of $s$ copies of $\Aut(T_\infty^q)$. 
This group has a subgroup $\left([C_q]^\infty\right)^s$, which is the direct product of $s$ copies of the 
permutation group given 
by the infinite iterated wreath product action of $C_q$ on $T_\infty^q$. If there are 
$s$ different polynomial maps $f_i(x)=x^q+c_i$ that satisfy the hypotheses of Theorem~\ref{p-theorem}, 
then it is easy to see that 
$G_\infty(\ba)$ embeds into $\left([C_q]^\infty\right)^s$. Thus we may rephrase Theorem~\ref{disjoint-theorem} 
as a finite index statement.

\begin{thm}\label{multitree theorem}
  Let $K$ be the function field of a smooth, projective, irreducible curve defined over $\Qbar$.  For
  $i =1, \dots, s$ let $f_i(x) = x^q + t_i \in K[x]$, where
  $t_i\notin \Qbar$, and suppose that $\alpha_i\in K$ are 
  neither periodic nor postcritical for $f$.  Suppose that there are no
  $i, j$ with the property that $(\alpha_i, \alpha_j)$ lies on a curve in $\A^2$
  that is periodic under the action of $(x,y) \mapsto (f_i(x), f_j(y))$. Then 
  \[[\left([C_q]^\infty\right)^s:G_\infty(\ba)]<\infty.\]
\end{thm}

\begin{proof}
The group $G_\infty(\ba)$ equals $\Gal(\prod_{i=1}^s K_\infty(f_i,\alpha_i)/K)$, which has 
finite index in the direct product $G_\infty(f_1,\alpha_1)\times\dots\times G_\infty(f_s,\alpha_s)$ by 
Theorem~\ref{disjoint-theorem} and basic Galois theory. This group in turn has finite index 
in $\left([C_q]^\infty\right)^s$ by applying Theorem~\ref{p-theorem} to each $G_\infty$ separately. 
\end{proof}

\section{The isotrivial case} \label{isotrivial case}

In this section we treat the case when the polynomial $f(x):=x^q+c$ is isotrivial, i.e, $c\in\Qbar$; also, we may assume the starting point $\beta\notin\Qbar$ since otherwise all the Galois groups corresponding to preimages of $\beta$ under iterates of $f$ would be trivial. So, throughout this section, we assume $K$ is the function field of a smooth, projective, irreducible curve defined over $\Qbar$. First we deal with the case when the polynomial $f$ is not PCF (see Proposition~\ref{easy isotrivial}); note the similarity in the statements of Proposition~\ref{easy isotrivial} and~\cite[Prop 12.1]{BT2}. In the special case $f$ is PCF, we will see in Proposition~\ref{PCF isotrivial} that we can answer even the case when the degree of the unicritical polynomial $f$ is not a prime power.

\begin{prop}\label{easy isotrivial}
Let $f(x):=x^q+c\in \Qbar[x]$ with $q=p^{r}$ a power of the prime number $p$ and let $\beta\in K\setminus \Qbar$.  
If $f$ is not PCF, then $[[C_q]^\infty:G_\infty(\beta)]=1$.
\end{prop}
\begin{proof}

The unique critical point of $f$ is $0$, which means that there are no integers $n>m>0$ 
such that $f^n(0)=f^m(0)$. In particular, $\cO_f(0)$ is infinite. 

Put $L=\Qbar(\beta)$.
Examining the Newton polygon for $f^n(x)-\beta$ with respect to the place of $L$ at infinity, we see that 
$f^n(x)-\beta$ is irreducible over $L$ for every $n$. Using the fact that the orbit $\cO_f(0)$ is infinite and 
non-repeating, it is easy to see that, for every $n$, the prime of $L$ 
generated by $\beta-f^n(0)$ satisfies Condition R for $\beta$ at $n$.
We are done by Propositions~\ref{main Galois} and~\ref{indexprop}.
\end{proof}

If $f$ is PCF of degree $d\ge 2$ (not necessarily a prime power), then we have a refined variation of the well-known fact that $G_\infty(\beta)$ has 
infinite index in $\Aut(T^d_\infty)$. The proof of Proposition~\ref{PCF isotrivial} does 
not explicitly use the fact that $f$ is isotrivial, but clearly, $c\in \Qbar$ since (as mentioned in Section~\ref{intro}) $f$ cannot be PCF otherwise.

\begin{prop}\label{PCF isotrivial}
Let $f(x):=x^d+c\in K[x]$ with $d\geq 2$ be any integer, and let $\beta\in K$. If $f$ is PCF, then $[[C_d]^\infty:G_\infty(\beta)]=\infty$.
\end{prop}
\begin{proof}
We argue as in~\cite[Thm 3.1]{RafeArborealSurvey}, which shows that the image of the arboreal 
Galois representation attached to any PCF rational map of degree $d\geq 2$ has infinite index 
in $\Aut(T^d_\infty)\cong [S_d]^\infty$. 
Using the fact that the abelianization of $G[H]$ is $G^{\text{ab}}\times H^{\text{ab}}$, we see that 
the abelianization of $[C_d]^n$ is $C_d^n$. Thus there are group homomorphisms from $[C_d]^n\to C_d^n$ 
for every $n\geq 1$,  
which induce a homomorphism $\tau: [C_d]^\infty\to C_d^{\N}$.

As $f$ is PCF,  by the main result of~\cite{CullinanHajir} (see also~\cite{BIJJ}), 
the extension $K_\infty(\beta)/K$ is unramified outside a finite set of primes $S$. By work of 
Ihara~\cite[Cor 10.1.6]{Neu}, the 
Galois group $G_{K,S}$ of the maximal extension of $K$ unramified outside of $S$ is topologically generated  
by the conjugacy classes of finitely many elements.
(Here it is important that $K$ has characteristic $0$ in order to ensure that all ramification is tame.) Of course, 
$G_\infty(\beta)$ is a quotient of $G_{K,S}$ and therefore shares this property. So $\tau(G_\infty(\beta))$ is topologically 
finitely generated in $C_d^\N\cong \tau([C_d]^\infty)$. We conclude that 
$[\tau([C_d]^\infty):\tau(G_\infty(\beta))]=\infty$, and therefore $[[C_d]^\infty:G_\infty(\beta)]=\infty$.
\end{proof}

\bibliographystyle{amsalpha}
\bibliography{QuadBib}

\end{document}